\documentclass[11pt]{amsart}
\usepackage{tikz-cd}
\usepackage{amscd}
\usepackage{amsmath, amssymb, comment}
\usepackage{amsfonts}
\usepackage{enumerate}
\usepackage{color}
\usepackage{url}
\newcommand{\de}{\partial}

\newcommand{\Ric}{\mathrm{Ric}}
\newcommand{\ov}[1]{\overline{#1}}

\newcommand{\vp}{\varphi}
\newcommand{\vol}{\mathrm{Vol}}
\newcommand{\diam}{\mathrm{diam}}
\newcommand{\ve}{\varepsilon}

\renewcommand{\leq}{\leqslant}
\renewcommand{\geq}{\geqslant}

\newcommand{\be}{\begin{equation}}
\newcommand{\ee}{\end{equation}}

\newcommand{\Hess}{\mathrm{Hess}}

\newcommand{\dvol}{\mathrm{dvol}}
\newcommand{\D}{\nabla}

\begin{document}
\newcounter{remark}
\newcounter{theor}
\setcounter{theor}{1}
\newtheorem{claim}{Claim}
\newtheorem{theorem}{Theorem}[section]
\newtheorem{lemma}[theorem]{Lemma}
\newtheorem{corollary}[theorem]{Corollary}
\newtheorem{conjecture}[theorem]{Conjecture}
\newtheorem{proposition}[theorem]{Proposition}
\newtheorem{question}{Question}[section]
\newtheorem{definition}[theorem]{Definition}
\newtheorem{remark}[theorem]{Remark}

\numberwithin{equation}{section}

\title{Spectral comparison results for the $N$-Bakry-Emery Ricci tensor}

\author{Jianchun Chu}
\address[Jianchun Chu]{School of Mathematical Sciences, Peking University, Yiheyuan Road 5, Beijing 100871, People's Republic of China}
\email{jianchunchu@math.pku.edu.cn}

\author{Zihang Hao}
\address[Zihang Hao]{School of Mathematical Sciences, Peking University, Yiheyuan Road 5, Beijing 100871, People's Republic of China}
\email{2301110014@pku.edu.cn}

\begin{abstract}
We establish the diameter and global weighted volume comparison when the $N$-Bakry-Emery Ricci tensor has a positive lower bound in the spectrum sense.
\end{abstract}

\maketitle

\markboth{Jianchun Chu and Zihang Hao}{Spectral comparison results for the $N$-Bakry-Emery Ricci tensor}

\section{Introduction}

Let $(M,g)$ be a complete $n$-dimensional Riemannian manifold and $\mathrm{dvol}_{g}$ be the volume element of $(M,g)$. For $f\in C^{\infty}(M)$, the triple $(M,g,e^{-f}\mathrm{dvol}_{g})$ is called a weighted manifold (or manifold with density), which was introduced by Lichnerowicz \cite{Lichnerowicz70,Lichnerowicz7172}. As a natural generalization of the Laplacian operator, the $f$-Laplacian operator is defined by
\[
\Delta_{f}u := e^{f}\mathrm{div}(e^{-f}\D u) = \Delta u-\langle\nabla u, \nabla f\rangle.
\]
It is clear that $\Delta_{f}$ is self-adjoint with respect to the weighted measure $e^{-f}\mathrm{dvol}_{g}$. As a natural generalization of the Ricci tensor, Lichnerowicz \cite{Lichnerowicz70} and Bakry-Emery \cite{BE85} introduced the tensor
\[
\Ric_{f} := \Ric + \Hess f,
\]
and extensively investigated the connection between $\Ric_{f}$ and properties of the underlying weighted manifold.
This tensor is usually referred to as the Bakry-Emery Ricci tensor in the literature. More generally, the $N$-Bakry-Emery Ricci tensor is defined by
\[
\Ric_f^N := \Ric + \Hess f -\frac{1}{N}df\otimes df.
\]
We will discuss some comparison results for the $N$-Bakry-Emery Ricci tensor below. For more related works, we refer the reader to \cite{BQ00,Perelman02,Lott03,Li05,Morgan09} and the references therein.

In the comparison geometry of Ricci curvature, the Bochner formula is very important, which gives the relationship between the $\Delta$ and $\Ric$:
\begin{equation}\label{Bochner formula Ricci}
\begin{split}
\frac{1}{2}\Delta|\nabla u|^{2}
= {} & |\Hess u|^{2}+\langle\nabla u,\nabla\Delta u\rangle+\Ric(\nabla u,\nabla u) \\
\geq {} & \frac{(\Delta u)^{2}}{n}+\langle\nabla u,\nabla\Delta u\rangle+\Ric(\nabla u,\nabla u).
\end{split}
\end{equation}
For the $N$-Bakry-Emery Ricci tensor, there is a counterpart:
\begin{equation}\label{Bochner formula N-Bakry-Emery Ricci}
\begin{split}
\frac{1}{2}\Delta_{f}|\nabla u|^{2}
= {} & |\Hess u|^{2}+\langle\nabla u,\nabla\Delta_{f} u\rangle+\Ric_{f}(\nabla u,\nabla u)\\
\geq {} & \frac{(\Delta_{f} u)^{2}}{N+n}+\langle\nabla u,\nabla\Delta_{f} u\rangle+\Ric_{f}^{N}(\nabla u,\nabla u).
\end{split}
\end{equation}
Comparing \eqref{Bochner formula Ricci} and \eqref{Bochner formula N-Bakry-Emery Ricci}, one may regard the Bochner formula for $\Ric_{f}^{N}$ of $n$-dimensional manifold as that for $\Ric$ of $(N+n)$-dimensional manifold.

\medskip

When $\Ric\geq(n-1)\lambda>0$, the Bonnet-Myers theorem shows that
\[
\diam(M,g)\leq\frac{\pi}{\sqrt{\lambda}},
\]
while the Bishop-Gromov volume comparison states that for $p\in M$, the volume ratio function
\[
\text{$r \mapsto \frac{\vol(B(p,r))}{\vol_\lambda^{n}(r)}$ is decreasing on $r\in\big(0,\pi/\sqrt{\lambda}\big]$},
\]
where $\vol_\lambda^{n}(\cdot)$ denotes the volume of geodesic ball in the $n$-dimensional sphere with radius $\lambda^{-\frac{1}{2}}$. In particular, the above implies the global volume comparison
\begin{equation}\label{global volume comparison}
\vol(M,g) \leq \vol_\lambda^{n}(\pi/\sqrt{\lambda}) \cdot \lim_{r\to0}\frac{\vol(B(p,r))}{\vol_\lambda^{n}(r)} = \lambda^{-\frac{n}{2}}\vol(\mathbb{S}^{n}).
\end{equation}

There has been great interest in extending the above comparison results to $N$-Bakry-Emery Ricci tensor setting. As in the above discussion of Bochner formula, when $N$ is finite, the $n$-dimensional weighted manifold with positive $N$-Bakry-Emery Ricci tensor looks like the $(N+n)$-dimensional manifold with positive Ricci tensor. However, when $N=\infty$, the situation becomes subtle. When $f$ is unbounded, the condition $\Ric_f \geq (n-1)\lambda$ does not even guarantee the compactness of $M$ (see \cite[Example 2.1 and 2.2]{WW09}). These examples shows that the boundedness assumption of $f$ (i.e. $F:=\|f\|_{L^{\infty}(M)}<\infty$) is necessary when $N=\infty$.

When $\Ric_f^N \geq (n-1)\lambda>0$ and $N\in(0,\infty)$, Qian \cite{Qian97} showed that
\[
\diam(M,g) \leq \sqrt{\frac{N+n-1}{n-1}} \cdot \frac{\pi}{\sqrt{\lambda}}.
\]
When $N=\infty$, diameter comparison results are listed below:\\[-2mm]
\begin{itemize}\setlength{\itemsep}{3mm}
\item Wei-Wylie \cite{WW09}: $\diam(M,g)\leq\big(1+\frac{4}{(n-1)\pi}F\big)\cdot\frac{\pi}{\sqrt{\lambda}}$;
\item Limoncu \cite{Limoncu12}: $\diam(M)\leq \sqrt{1+\frac{2\sqrt{2}}{n-1}F}\cdot\frac{\pi}{\sqrt{\lambda}}$; \\[-1mm]
\item Tadano \cite{Tadano16}: $\diam(M)\leq \sqrt{1+\frac{8}{(n-1)\pi}F}\cdot\frac{\pi}{\sqrt{\lambda}}$. \\[-1mm]
\end{itemize}
On the other hand, Bakry-Qian \cite{BQ05} proved that when $\Ric_{f}^{N}\geq(n-1)\lambda>0$, then for $p\in M$, the weighted volume ratio function
\begin{equation}\label{weighted volume ratio function N+n}
\text{$r \mapsto \frac{\vol_{f}(B(p,r))}{\vol_\lambda^{N+n}(r)}$ is decreasing on $r\in\left(0,\sqrt{\frac{N+n-1}{n-1}} \cdot \frac{\pi}{\sqrt{\lambda}}\right]$}.
\end{equation}
When $N=\infty$, Wei-Wylie \cite{WW09} proved that for $p\in M$,
\begin{equation}\label{weighted volume ratio function n+4F}
\text{$r \mapsto \frac{\vol_{f}(B(p,r))}{\vol_\lambda^{n+4F}(r)}$ is decreasing on $r\in\big(0,\pi/4\sqrt{\lambda}\big]$}.
\end{equation}
Note that one does not obtain the global weighted volume comparison (i.e. the weighted version of \eqref{global volume comparison}) since the limit in \eqref{weighted volume ratio function N+n} or \eqref{weighted volume ratio function n+4F} blows up when $r\to0$. To the best of our knowledge, such global weighted volume comparison is still lacking.

\medskip

Now let us go back to the $n$-dimenisnal Riemannian manifolds with $\Ric\geq(n-1)\lambda>0$. Following the isoperimetric profile method from Bray \cite{Bray97}, Bray-Gui-Liu-Zhang \cite{BGLZ19} gave an alternative proof of the global volume comparison $\vol(M,g)\leq\lambda^{-n/2}\vol(\mathbb{S}^{n})$. Following the ideas developed in \cite{Bray97,BGLZ19} and $\mu$-bubble method from Gromov \cite{Gromov19}, Antonelli-Xu \cite{AX17} proved the global volume comparison and diameter comparison in the spectrum sense (see also \cite{Xu23,CLMS24}).

\begin{theorem}[Theorem 1 and Lemma 1 of \cite{AX17}]\label{Antonelli-Xu theorem}
Let $(M,g)$ be a complete $n$-dimensional Riemannian manifold with $n\geq3$, and $\gamma$ be a constant such that
\[
 0 \leq \gamma \leq \frac{n-1}{n-2}.
\]
Suppose that there exists a positive function $u\in C^{\infty}(M)$ such that
\[
u\,\Ric(x)-\gamma\Delta u \geq (n-1)\lambda u
\]
for some $\lambda>0$. Here 
\[
\Ric(x) := \inf_{v\in T_xM,\,|v|=1}\Ric(v,v)
\] 
denotes the smallest eigenvalue of the Ricci tensor at $x$. If
\[
0 < \inf_{M}u \leq \sup_{M} u < \infty,
\]
then the following comparison results hold:
\begin{enumerate}\setlength{\itemsep}{1mm}
\item Diameter comparison:
\[
\diam(M,g) \leq
\left(\frac{\sup_{M}u}{\inf_{M}u}\right)^{\frac{n-3}{n-1}\gamma} \cdot \frac{\pi}{\sqrt{\lambda}}.
\]

\item Global volume comparison:
\[
\vol(M,g) \leq \lambda^{-n/2}\vol(\mathbb{S}^{n}).
\]
\end{enumerate}
\end{theorem}

Motivated by Theorem \ref{Antonelli-Xu theorem}, we aim to establish analogous results for the $N$-Bakry-Emery Ricci tensor by adapting the approach of Antonelli-Xu \cite{AX17}.

\begin{theorem}\label{main theorem}
Let $(M,g)$ be a complete $n$-dimensional Riemannian manifold with $n\geq3$, and $N$, $\gamma$ be two constants such that
\[
N \in (-\infty,-(n-1)) \cup (0,+\infty), \ \ \
0 \leq \gamma \leq \frac{N+n-1}{N+n-2}.
\]
Suppose that there exists a positive function $u\in C^{\infty}(M)$ such that
\[
u\,\Ric_{f}^{N}(x)-\gamma\Delta_{f}u \geq (n-1)\lambda u
\]
for some $f\in C^{\infty}(M)$ and $\lambda>0$. Here 
\[
\Ric_f^N(x) := \inf_{v\in T_xM,\,|v|=1}\Ric_{f}^N(v,v)
\] 
denotes the smallest eigenvalue of the $N$-Bakry-Emery Ricci tensor at $x$. If
\[
0 < \inf_{M}u \leq \sup_{M} u < \infty, \ \ F = \|f\|_{C^{0}(M)} < \infty,
\]
then the following comparison results hold:
\begin{enumerate}\setlength{\itemsep}{1mm}
\item Diameter comparison:
\[
\diam(M,g) \leq
\begin{cases}
\left(\frac{\sup_{M}u}{\inf_{M}u}\right)^{\frac{N+n-3}{N+n-1}\gamma}
\cdot \sqrt{\frac{N+n-1}{n-1}} \cdot \frac{\pi}{\sqrt{\lambda}} & \mbox{\text{when $N>0$}}, \\[2mm]
\left(\frac{\sup_{M}u}{\inf_{M}u}\right)^{\frac{n-3}{n-1}\gamma}
\cdot e^{\frac{2}{n-1}F} \cdot \frac{\pi}{\sqrt{\lambda}} & \mbox{\text{when $N<-(n-1)$}}.
\end{cases}
\]

\item Global weighted volume comparison:
\[
\vol_{f}(M,g) \leq e^{\frac{(n+1)(3n-1)}{n(n-1)}F} \cdot \lambda^{-\frac{n}{2}}\vol(\mathbb{S}^{n}).
\]
\end{enumerate}
\end{theorem}

\begin{remark}
To the best of our knowledge, even when $u\equiv 1$, (2) of Theorem \ref{main theorem} seems to be new, which gives a global weighted volume comparison for $(M,g,e^{-f}\dvol_{g})$ with $\Ric_{f}^{N}\geq(n-1)\lambda>0$.
\end{remark}

\medskip

{\bf Acknowledgements.}
The authors would like to thank Kai Xu, Jintian Zhu and Shihang He for useful discussions. The first-named author was partially supported by National Key R\&D Program of China 2023YFA1009900 and NSFC grant 12271008. The second-named author was partially supported by National Key R\&D Program of China 2023YFA1009900.

\section{Proof of diameter comparison}

Suppose that $\Omega_-\subset \Omega_+\subset M$ be two domains such that $\de\Omega_{-}\neq\emptyset$, $\de\Omega_{+}\neq\emptyset$ and $\overline{\Omega_+}\setminus\Omega_-$ is compact. Let $h\in C^\infty(\Omega_{+}\setminus\overline{\Omega_-})$ be a function satisfying
\[
\text{$\lim_{x\to \de\Omega_-}h(x) = +\infty$, $\lim_{x\to \de\Omega_+}h(x) = -\infty$ uniformly.}
\]
Fix a constant $k$ and a domain $\Omega_0$ such that $\Omega_-\Subset \Omega_0\Subset \Omega_+$. For a set $\Omega$ with finite perimeter, we consider the functional
\[
E(\Omega)=\int_{\de^*\Omega}u^{\gamma}e^{-f}-\int(\chi_{\Omega}-\chi_{\Omega_0})hu^\alpha e^{-(k+1)f},
\]
where $\de^*\Omega$ denotes the reduced boundary of $\Omega$. By the similar argument of \cite[Proposition 2.1]{Zhu21}, the functional $E$ must have a minimizer $\Omega$ such that $\Omega_{-}\Subset\Omega\Subset\Omega_{+}$. The classical geometric measure theory (see e.g. \cite[Theorem 27.5, Theorem 28.1]{Maggi12}) shows that $\de\Omega$ is smooth when $n\leq 7$ while the singular part of $\de\Omega$ has Hausdorff dimension no more than $n-8$.

\subsection{The case $3\leq n\leq7$}

For any $\vp\in C^{\infty}(M)$, let $\{\Omega_{t}\}_{t\in(-\ve,\ve)}$ be a smooth family of sets satisfying
\begin{itemize}\setlength{\itemsep}{2mm}
\item $\Omega_{0}=\Omega$;
\item $\frac{\de \Omega_{t}}{\de t}=\vp\nu$, where $\nu$ denotes the outer unit normal of $\de\Omega_{t}$.
\end{itemize}
We compute the first variation:
\[
0 = \frac{\mathrm{d}}{\mathrm{d}t}E(\Omega_t)\Big|_{t=0}
= \int_{\partial \Omega}(H+\gamma u^{-1}u_\nu-f_\nu-hu^{\alpha-\gamma}e^{-kf})u^\gamma e^{-f}\vp.
\]
Since $\vp$ is arbitrary, then we obtain
\begin{equation}\label{H}
H = f_\nu+hu^{\alpha-\gamma}e^{-kf}-\gamma u^{-1}u_\nu.
\end{equation}
We next compute the second variation:
\[
\begin{split}
0 \leq {} & \frac{\mathrm{d}^{2}}{\mathrm{d}t^{2}}E(\Omega_t)\Big|_{t=0} \\
= {} & \int_{\partial \Omega}\Big(-\Delta_{\partial\Omega}\varphi-|\mathrm{II}|^2\varphi-\Ric(\nu,\nu)\varphi-\gamma u^{-2}u_\nu^2\varphi \\
& \quad \quad \ \, +\gamma u^{-1}\varphi(\Delta u-\Delta_{\partial \Omega}u-Hu_\nu)-\gamma u^{-1}\langle\nabla_{\partial\Omega}u,\nabla_{\partial\Omega}\varphi\rangle \\[1mm]
& \quad \quad \ \, -\mathrm{Hess} f(\nu,\nu)\vp + \langle\nabla_{\partial\Omega} f,\nabla_{\partial \Omega}\varphi\rangle-h_{\nu}u^{\alpha-\gamma}e^{-kf}\varphi\\
& \quad \quad \ \, -(\alpha-\gamma)hu^{\alpha-\gamma-1}u_\nu e^{-kf} \varphi+k hu^{\alpha-\gamma}e^{-kf}f_{\nu}\varphi\Big) u^\gamma e^{-f}\varphi.
\end{split}
\]
By the definition of $\Delta_{f}$ and $\Ric_{f}^{N}$,
\begin{equation}\label{Ric Delta u Hess f}
\begin{split}
& -\Ric(\nu,\nu)\varphi+\gamma u^{-1}\varphi\Delta u-\mathrm{Hess} f(\nu,\nu)\vp \\[1mm]
= {} & -\big(\Ric_{f}^{N}(\nu,\nu)+\frac{1}{N}f_{\nu}^{2}\big)\varphi+\gamma u^{-1}\varphi\big(\Delta_{f}u+\langle\nabla u,\nabla f\rangle\big) \\
= {} & \big(\gamma\Delta_{f}u-u\Ric_{f}^{N}(\nu,\nu)\big)u^{-1}\varphi-\frac{1}{N}f_{\nu}^{2}\vp+\gamma u^{-1}\varphi\langle\nabla u,\nabla f\rangle.
\end{split}
\end{equation}
It then follows that
\[
\begin{split}
0 \leq \int_{\partial \Omega}
\Big(& -\Delta_{\partial\Omega}\varphi-|\mathrm{II}|^2\varphi+\big(\gamma\Delta_{f}u-u\Ric_{f}^{N}(\nu,\nu)\big)u^{-1}\varphi-\frac{1}{N}f_{\nu}^{2}\vp \\
& +\gamma u^{-1}\varphi\langle\nabla u,\nabla f\rangle-\gamma u^{-2}u_\nu^2\varphi+\gamma u^{-1}\varphi(-\Delta_{\partial \Omega}u-Hu_\nu) \\[1mm]
& -\gamma u^{-1}\langle\nabla_{\partial\Omega}u,\nabla_{\partial\Omega}\varphi\rangle
+\langle\nabla_{\partial\Omega} f,\nabla_{\partial \Omega}\varphi\rangle-h_{\nu}u^{\alpha-\gamma}e^{-kf}\varphi\\
& -(\alpha-\gamma)hu^{\alpha-\gamma-1}u_\nu e^{-kf} \varphi+khu^{\alpha-\gamma}e^{-kf}f_{\nu}\varphi\Big) u^\gamma e^{-f}\varphi.
\end{split}
\]

\begin{lemma}\label{h inequality}
When $\vp=u^{-\gamma}$ and $\alpha=\frac{kN+2}{N+n-1}\gamma$, we have
\[
\begin{split}
0 \leq \int_{\partial \Omega}&\Big( -\frac{\big(4k-(n-1)k^2\big)N+4}{4(N+n-1)}u^{2\alpha-2\gamma}e^{-kf}h^{2} \\
& +|\nabla h|u^{\alpha-\gamma}-(n-1)\lambda e^{kf}
\Big)u^{-\gamma}e^{-(k+1)f}.
\end{split}
\]
\end{lemma}

\begin{proof}
By $\vp=u^{-\gamma}$ and $\gamma\geq0$, we see that
\begin{equation}\label{Delta vp Delta u nabla u nabla vp}
-\Delta_{\partial\Omega}\varphi-\gamma u^{-1}\varphi\Delta_{\partial \Omega}u
-\gamma u^{-1}\langle\nabla_{\partial\Omega}u,\nabla_{\partial\Omega}\varphi\rangle
= -\gamma u^{-\gamma-2}|\nabla_{\de\Omega}u|^{2} \leq 0
\end{equation}
and
\begin{equation}\label{nabla u nabla f nabla f nabla vp}
\begin{split}
& \gamma u^{-1}\varphi\langle\nabla u,\nabla f\rangle+\langle\nabla_{\partial\Omega} f,\nabla_{\partial \Omega}\varphi\rangle \\
= {} & \gamma u^{-\gamma-1}\big(\langle\nabla u,\nabla f\rangle-\langle\nabla_{\partial \Omega}u,\nabla_{\partial\Omega} f\rangle\big)
=  \gamma u^{-\gamma-1}u_{\nu}f_{\nu}.
\end{split}
\end{equation}
Combining the above with $|\mathrm{II}|^2\geq \frac{H^2}{n-1}$ and $\gamma \Delta_f u-u\Ric_f^N \leq -(n-1)\lambda u$,
\[
\begin{split}
0 \leq \int_{\partial \Omega}
\Big(& -\frac{H^2}{n-1}u^{-\gamma}-(n-1)\lambda u^{-\gamma}-\frac{1}{N}f_{\nu}^{2}u^{-\gamma}+\gamma u^{-\gamma-1}u_{\nu}f_{\nu}\\
& -\gamma u^{-\gamma-2}u_\nu^2-\gamma u^{-\gamma-1}Hu_\nu -h_{\nu}u^{\alpha-\gamma}e^{-kf}u^{-\gamma}\\
& -(\alpha-\gamma)hu^{\alpha-\gamma-1}u_\nu e^{-kf}u^{-\gamma}+khu^{\alpha-\gamma}e^{-kf}f_{\nu}u^{-\gamma}\Big)e^{-f} \\
\leq \int_{\partial \Omega}
\Big(& -\frac{H^2}{n-1}-(n-1)\lambda-\frac{1}{N}f_{\nu}^{2}+\gamma u^{-1}u_{\nu}f_{\nu}\\
& -\gamma u^{-2}u_\nu^2-\gamma u^{-1}Hu_\nu +|\nabla h|u^{\alpha-\gamma}e^{-kf}\\
& -(\alpha-\gamma)hu^{\alpha-\gamma-1}u_\nu e^{-kf}+khu^{\alpha-\gamma}e^{-kf}f_{\nu}\Big)u^{-\gamma}e^{-f}.
\end{split}
\]
Set $X=hu^{\alpha-\gamma}e^{-kf}$ and $Y=u^{-1}u_{\nu}$. By \eqref{H}, we have $H=f_{\nu}+X-\gamma Y$. Direct calculation shows
\[
\begin{split}
0 \leq \int_{\de\Omega}\bigg[& -\frac{(f_{\nu}+X-\gamma Y)^2}{n-1}-(n-1)\lambda-\frac{1}{N}f_{\nu}^2 \\
& +\gamma Yf_\nu-\gamma Y^2  -\gamma Y(f_\nu+X-\gamma Y)\\
&+|\nabla h|u^{\alpha-\gamma}e^{-kf}+(\gamma-\alpha)XY+kXf_\nu  \bigg] u^{-\gamma}e^{-f} \\
= \int_{\partial \Omega}
\bigg[& -\Big(\frac{1}{n-1}+\frac{1}{N}\Big)f_{\nu}^2+\Big(k-\frac{2}{n-1}\Big)Xf_{\nu}+\frac{2\gamma}{n-1}Yf_{\nu} \\
& -\frac{X^{2}}{n-1}+\Big(\frac{2\gamma}{n-1}-\alpha\Big)XY+\Big(\frac{n-2}{n-1}\gamma^2-\gamma\Big)Y^2 \\
& +|\nabla h|u^{\alpha-\gamma}e^{-kf}-(n-1)\lambda\bigg]u^{-\gamma}e^{-f}.
\end{split}
\]
It is clear that
\[
\begin{split}
& -\Big(\frac{1}{n-1}+\frac{1}{N}\Big)f_{\nu}^2+\Big(k-\frac{2}{n-1}\Big)Xf_{\nu}+\frac{2\gamma}{n-1}Yf_{\nu} \\
\leq {} & \frac{1}{4}\cdot\Big(\frac{1}{n-1}+\frac{1}{N}\Big)^{-1}\left[\Big(k-\frac{2}{n-1}\Big)X+\frac{2\gamma}{n-1}Y\right]^{2} \\
= {} & \frac{N\big((n-1)k-2\big)^{2}}{4(N+n-1)(n-1)}X^{2}+\frac{N\big((n-1)k-2\big)\gamma}{(N+n-1)(n-1)}XY \\
& +\frac{N\gamma^{2}}{(N+n-1)(n-1)}Y^{2}
\end{split}
\]
and so
\[
\begin{split}
0 \leq \int_{\de\Omega}\bigg[ & \frac{\big((n-1)k^2-4k\big)N-4}{4(N+n-1)}X^2+\Big(\frac{kN+2}{N+n-1}\gamma-\alpha\Big)XY \\
& +\Big(\frac{N+n-2}{N+n-1}\gamma^2-\gamma\Big)Y^2+|\nabla h|u^{\alpha-\gamma}e^{-kf}-(n-1)\lambda\bigg]u^{-\gamma}e^{-f}.
\end{split}
\]
Since $\alpha=\frac{kN+2}{N+n-1}\gamma$ and $0\leq \gamma \leq\frac{N+n-1}{N+n-2}$, then
\[
0 \leq \int_{\partial \Omega}\bigg[-\frac{\big(4k-(n-1)k^2\big)N+4}{4(N+n-1)}X^2+|\nabla h|u^{\alpha-\gamma}e^{-kf}-(n-1)\lambda\bigg]u^{-\gamma}e^{-f}.
\]
Recalling $X=hu^{\alpha-\gamma}e^{-kf}$, we obtain
\[
\begin{split}
0 \leq {}  \int_{\partial \Omega}\Big(& -\frac{\big(4k-(n-1)k^2\big)N+4}{4(N+n-1)}u^{2\alpha-2\gamma}e^{-2kf}h^{2} \\
& +|\nabla h|u^{\alpha-\gamma}e^{-kf}-(n-1)\lambda
\Big)u^{-\gamma}e^{-f} \\
= {}  \int_{\partial \Omega}\Big(& -\frac{\big(4k-(n-1)k^2\big)N+4}{4(N+n-1)}u^{2\alpha-2\gamma}e^{-kf}h^{2} \\
& +|\nabla h|u^{\alpha-\gamma}-(n-1)\lambda e^{kf}
\Big)u^{-\gamma}e^{-(k+1)f},
\end{split}
\]
as required.
\end{proof}

\begin{lemma}\label{h lemma}
Let $(M,g)$ be a complete Riemannian manifold and $C$, $D$ be two positive constants. If $\diam(M,g)>\frac{\pi}{\sqrt{CD}}$, then there exist two domains $\Omega_{-}\subset\Omega_{+}\subset M$ and function $h\in C^{\infty}(\Omega_{+}\setminus\ov{\Omega_{-}})$ such that
\begin{itemize}\setlength{\itemsep}{1mm}
\item $\de\Omega_{-}\neq\emptyset$, $\de\Omega_{+}\neq\emptyset$ and $\overline{\Omega_+}\setminus\Omega_-$ is compact;
\item $\lim_{x\to\de\Omega_{-}}h(x)=+\infty$ and $\lim_{x\to\de\Omega_{+}}h(x)=-\infty$ uniformly;
\item $|\nabla h|<Ch^{2}+D$.
\end{itemize}
\end{lemma}

\begin{proof}
This lemma is proved in the argument of \cite[Lemma 1]{AX17}.
\end{proof}

Now we are in a position to prove diameter comparison when $3\leq n\leq7$.

\begin{proof}[Proof of (1) in Theorem \ref{main theorem} when $3\leq n\leq7$]
When the constant $k$ satisfies
\[
\text{$\frac{kN+2}{N+n-1}\leq1$ and $\frac{\big(4k-(n-1)k^2\big)N+4}{4(N+n-1)}>0$},
\]
we have $\alpha=\frac{kN+2}{N+n-1}\gamma\leq\gamma$ and define two positive constants by
\[
C = \frac{\big(4k-(n-1)k^2\big)N+4}{4(N+n-1)}\cdot\frac{\left(\sup_{M}u\right)^{2\alpha-2\gamma}}{\left(\inf_{M}u\right)^{\alpha-\gamma}} \cdot e^{-|k|F}
\]
and
\[
D = (n-1)\lambda \left(\inf_{M}u\right)^{\gamma-\alpha}e^{-|k|F}.
\]
Then Lemma \ref{h inequality} shows
\[
\begin{split}
0 \leq {} & \int_{\partial \Omega}\Big(|\nabla h|u^{\alpha-\gamma}-C\big(\inf_{M}u\big)^{\alpha-\gamma}h^{2}-D\big(\inf_{M}u\big)^{\alpha-\gamma}\Big)u^{-\gamma}e^{-(k+1)f} \\
\leq {} & \int_{\partial \Omega}\Big(|\nabla h|-Ch^{2}-D\Big)u^{-\alpha-2\gamma}e^{-(k+1)f}.
\end{split}
\]
If $\diam(M,g)>\frac{\pi}{\sqrt{CD}}$, then we choose $h$ as in Lemma \ref{h lemma} and obtain a contradiction. Thus $\diam(M,g)\leq\frac{\pi}{\sqrt{CD}}$ and so
\[
\diam(M,g)\leq \left(\frac{\sup_{M}u}{\inf_{M}u}\right)^{\gamma-\alpha}
\cdot \sqrt{\frac{4(N+n-1)}{\big(4k-(n-1)k^2\big)N+4}}\cdot\sqrt{\frac{e^{2|k|F}}{n-1}} \cdot \frac{\pi}{\sqrt{\lambda}}.
\]
When $N>0$, we choose $k=0$. Then $\alpha=\frac{2\gamma}{N+n-1}$ and
\[
\diam(M,g) \leq \left(\frac{\sup_{M}u}{\inf_{M}u}\right)^{\frac{N+n-3}{N+n-1}\gamma}
\cdot \sqrt{\frac{N+n-1}{n-1}} \cdot \frac{\pi}{\sqrt{\lambda}}.
\]
When $N<-(n-1)$, we choose $k=\frac{2}{n-1}$. Then $\alpha=\frac{2\gamma}{n-1}$ and
\[
\diam(M,g) \leq \left(\frac{\sup_{M}u}{\inf_{M}u}\right)^{\frac{n-3}{n-1}\gamma}
\cdot e^{\frac{2}{n-1}F} \cdot \frac{\pi}{\sqrt{\lambda}}.
\]
\end{proof}

\subsection{The case $n\geq8$}
In this case, $\de\Omega$ may have singularities. The diameter comparison can be proved by modifying the argument (of the case $3\leq n\leq 7$) verbatim as in \cite[Proof of Lemma 1 ($n\geq8$) in Appendix A]{AX17}.

\section{Global weighted volume comparison}

To establish the global weighted volume comparison, we first recall the following ODE comparison (see e.g. \cite[Lemma 4]{AX17}).
\begin{lemma}\label{ODE comparison}
Let $V_{1}\in(0,+\infty)$ be a constant and $J:[0,V_{1})\to\mathbb{R}$ be a continuous function such that
\begin{itemize}\setlength{\itemsep}{1mm}
\item[(a)] $J(0)=0$ and $J(v)>0$ for $v\in(0,V_{1})$;
\item[(b)] $\limsup_{v\to0^{+}}v^{-\frac{n-1}{n}}J(v)\leq n\vol(\mathbb{B}^{n})^{\frac{1}{n}}$.
\item[(c)] $J''J\leq-\frac{(J')^{2}}{n-1}-(n-1)\Lambda$ in the viscosity sense on $(0,V_{1})$ for some positive constant $\Lambda$;
\end{itemize}
Then $V_{1}\leq\Lambda^{-\frac{n}{2}}\vol(\mathbb{S}^{n})$.
\end{lemma}

We assume without loss of generality that $\inf_{M}u=1$. Set
\[
\alpha = \frac{2\gamma}{n-1}, \ \ k = \frac{2}{n-1}.
\]
By (1) of Theorem \ref{main theorem}, $M$ is compact and so $V_{0}:=\int_{M}u^\alpha e^{-(k+1)f}$ is finite. Define the weighted isoperimetric profile:
\[
I(v) = \inf\left\{ \int_{\de^{*}E}u^\gamma e^{-f}: \text{$E\Subset M$ has finite perimeter, and $\int_{E}u^\alpha e^{-(k+1)f}=v$} \right\}
\]
for all $v\in[0,V_{0})$, where $\de^{*}E$ denotes the reduced boundary of $E$. It is clear that
\begin{equation}\label{ODE comparison condition a}
\text{$I(0)=0$ and $I(v)>0$ for $v\in(0,V_{0})$}.
\end{equation}
Using an argument similar to \cite[Proposition 5.3]{CLMS24}, $I$ is a continuous function on $[0,V_{0})$.

Since $\inf_{M}u=1$ and $M$ is compact, then there exists a point $x_{0}\in M$ such that $u(x_{0})=1$. When $r\ll1$, we have the asymptotical expansion:
\[
\int_{B_{r}(x_{0})}e^{-(k+1)f} = \vol(\mathbb{B}^{n})r^{n}e^{-(k+1)f(x_{0})}+O(r^{n+1})
\]
and
\[
\int_{\de B_{r}(x_{0})} e^{-f} = n\vol(\mathbb{B}^{n})r^{n-1}e^{-f(x_{0})}+O(r^{n}).
\]
By the definition of $I$, when $v\ll1$, we have
\[
\begin{split}
I(v) \leq {} & n\vol(\mathbb{B}^{n})^{\frac{1}{n}}v^{\frac{n-1}{n}}e^{\frac{(n-1)(k+1)-n}{n}f(x_{0})}+o(v^{\frac{n-1}{n}}) \\
= {} & n\vol(\mathbb{B}^{n})^{\frac{1}{n}}v^{\frac{n-1}{n}}e^{\frac{1}{n}f(x_{0})}+o(v^{\frac{n-1}{n}}) \\
\leq {} & n\vol(\mathbb{B}^{n})^{\frac{1}{n}}v^{\frac{n-1}{n}}e^{\frac{1}{n}F}+o(v^{\frac{n-1}{n}}),
\end{split}
\]
where we used $k=\frac{2}{n-1}$ in the second line. This shows
\begin{equation}\label{ODE comparison condition b}
\displaystyle{\limsup_{v\to0^{+}}}\,v^{-\frac{n-1}{n}}I(v)\leq n\vol(\mathbb{B}^{n})^{\frac{1}{n}}e^{\frac{1}{n}F}.
\end{equation}

For each $v_{0}\in(0,V_{0})$ and let $E$ be the volume-constrained minimizer of $v_{0}$, i.e.
\[
\int_{E}u^\alpha e^{-(k+1)f} = v_{0}, \ \ \ \int_{\de^{*}E}u^\gamma e^{-f} = I(v_{0}).
\]
The classical geometric measure theory (see e.g. \cite[Section 3.10]{Morgan03}) shows that $\de E$ is smooth when $n\leq 7$ while the singular part of $\de E$ has Hausdorff dimension no more than $n-8$.

\subsection{The case $3\leq n\leq7$}

\begin{lemma}\label{ODE comparison condition c}
The function $I$ satisfies
\[
\text{$I''I\leq-\frac{(I')^{2}}{n-1}-(n-1)\lambda e^{-2kF}$ in the viscosity sense on $(0,V_{0})$.}
\]
\end{lemma}

\begin{proof}
Fix $v_{0}\in(0,V_{0})$ and let $E$ be the above volume-constrained minimizer of $v_{0}$. For any $\vp\in C^{\infty}(M)$, let $\{E_{t}\}_{t\in(-\ve,\ve)}$ be a smooth family of sets satisfying
\begin{itemize}\setlength{\itemsep}{2mm}
\item $E_{0}=E$;
\item $\frac{\de E_{t}}{\de t}=\vp\nu$, where $\nu$ denotes the outer unit normal of $\de E_{t}$.
\end{itemize}
Write
\[
V(t) = \int_{E_{t}}u^\alpha e^{-(k+1)f}, \ \ \
A(t) = \int_{\de E_{t}}u^\gamma e^{-f}.
\]
Direct calculation shows
\[
V'(0) = \int_{\de E}u^\alpha e^{-(k+1)f}\vp,
\]
\[
V''(0) = \int_{\de E}\big(H+\alpha u^{-1}u_{\nu}-(k+1)f_{\nu}\big)u^\alpha e^{-(k+1)f}\vp^{2}+u^{\alpha}e^{-(k+1)f}\vp\vp_{\nu}
\]
and
\[
A'(0) = \int_{\de E}(H-f_{\nu}+\gamma u^{-1}u_\nu)u^\gamma e^{-f}\vp,
\]
\[
\begin{split}
A''(0) = \int_{\de E}\Big( & -\Delta_{\de E}\vp-\Ric(\nu,\nu)\vp-|\mathrm{II}|^{2}\vp-\mathrm{Hess}f(\nu,\nu)\vp \\
& +\langle\nabla_{\de E}f,\nabla_{\de E}\vp\rangle +\gamma u^{-1}\vp(\Delta u-\Delta_{\de E}u-Hu_\nu)\\
&-\gamma u^{-2}u_\nu^2\vp-\gamma u^{-1}\langle\nabla_{\de\Omega}u,\nabla_{\de\Omega}\vp\rangle\Big)u^\gamma e^{-f}\vp\\[1mm]
& +\big(H-f_{\nu}+\gamma u^{-1}u_\nu \big)\big(\gamma u^{-1}u_\nu \vp-f_\nu \vp+\vp_\nu+H\vp\big)u^{\gamma}e^{-f}\vp.
\end{split}
\]
For later use, we simplify the expression of $A''(0)$. By the same calculation of \eqref{Ric Delta u Hess f}, we have
\[
\begin{split}
& -\Ric(\nu,\nu)\varphi+\gamma u^{-1}\varphi\Delta u-\mathrm{Hess} f(\nu,\nu)\vp \\
= {} & \big(\gamma\Delta_{f}u-u\Ric_{f}^{N}(\nu,\nu)\big)u^{-1}\varphi-\frac{1}{N}f_{\nu}^{2}\vp+\gamma u^{-1}\varphi\langle\nabla u,\nabla f\rangle.
\end{split}
\]
Then
\[
\begin{split}
A''(0) = \int_{\de E}\Big( & -\Delta_{\de E}\vp-|\mathrm{II}|^{2}\vp+\big(\gamma\Delta_{f}u-u\Ric_{f}^{N}(\nu,\nu)\big)u^{-1}\varphi-\frac{1}{N}f_{\nu}^{2}\vp \\
& +\gamma u^{-1}\varphi\langle\nabla u,\nabla f\rangle+\langle\nabla_{\de E}f,\nabla_{\de E}\vp\rangle-\gamma u^{-2}u_\nu^2\vp \\[1mm]
& +\gamma u^{-1}\vp(-\Delta_{\de E}u-Hu_\nu)-\gamma u^{-1}\langle\nabla_{\de\Omega}u,\nabla_{\de\Omega}\vp\rangle \\
& +(H-f_{\nu}+\gamma u^{-1}u_\nu \big)\big(\gamma u^{-1}u_\nu \vp-f_\nu \vp+\vp_\nu+H\vp)\Big)u^\gamma e^{-f}\vp.
\end{split}
\]

\medskip

Since $E$ is a volume-constrained minimizer, we know that if
\[
V'(0) = \int_{\de E}u^\alpha e^{-(k+1)f}\vp = 0,
\]
then
\[
A'(0) = \int_{\de E}(H-f_{\nu}+\gamma u^{-1}u_\nu)u^{\gamma-\alpha}e^{kf} \cdot u^\alpha e^{-(k+1)f}\vp = 0.
\]
This shows
\[
\text{$(H-f_{\nu}+\gamma u^{-1}u_\nu)u^{\gamma-\alpha}e^{kf} = \mathrm{constant}$ on $E$.}
\]

\medskip

Next, we choose $\vp=u^{-\gamma}$ and so
\[
V'(0) = \int_{\de E}u^{\alpha-\gamma} e^{-(k+1)f} > 0.
\]
Then near $0$, $V(t)$ has the inverse function and we define $G(v)=A(V^{-1}(v))$. Observe that
\[
\text{$G(v_{0})=I(v_{0})$ and $G(v)\geq I(v)$ near $v_{0}$}.
\]
Thus $G$ is a upper barrier of $I$ at $v_{0}$. To prove Lemma \ref{ODE comparison condition c}, it then suffices to show
\begin{equation}\label{goal G}
G(v_{0})G''(v_{0}) \leq -\frac{G'^{2}(v_{0})}{n-1}-(n-1)\lambda e^{-2kF}.
\end{equation}
Using the chain rule and the formula for the derivative of inverse function, we compute
\begin{equation}\label{G'}
G'(v_{0}) = A'(0)(V^{-1})'(v_{0}) = \frac{A'(0)}{V'(0)} = (H-f_{\nu}+\gamma u^{-1}u_\nu)u^{\gamma-\alpha}e^{kf}
\end{equation}
and
\[
G''(v_{0}) = \frac{A''(0)}{V'^{2}(0)}-\frac{A'(0)V''(0)}{V'^{3}(0)}
= \frac{A''(0)-G'(v_{0})V''(0)}{V'^{2}(0)}.
\]
We split the argument of \eqref{goal G} into two steps:

\bigskip
\noindent
{\bf Step 1.} Compute $V''(0)$ and $A''(0)$.
\bigskip

It is clear that
\[
V''(0) = \int_{\de E}\Big(H-(k+1)f_{\nu}+(\alpha-\gamma)u^{-1}u_{\nu}\Big)u^{\alpha-2\gamma}e^{-(k+1)f}.
\]
For $A''(0)$, the same calculations of \eqref{Delta vp Delta u nabla u nabla vp} and \eqref{nabla u nabla f nabla f nabla vp} show
\[
-\Delta_{\partial E}\varphi-\gamma u^{-1}\varphi\Delta_{\partial E}u
-\gamma u^{-1}\langle\nabla_{\partial E}u,\nabla_{\partial E}\varphi\rangle
= -\gamma u^{-\gamma-2}|\nabla_{\de E}u|^{2} \leq 0
\]
and
\[
\gamma u^{-1}\varphi\langle\nabla u,\nabla f\rangle+\langle\nabla_{\partial E}f,\nabla_{\partial E}\varphi\rangle
= \gamma u^{-\gamma-1}u_{\nu}f_{\nu}.
\]
Combining the above with $|\mathrm{II}|^2\geq \frac{H^2}{n-1}$ and $\gamma \Delta_f u-u\Ric_f^N \leq -(n-1)\lambda u$, we obtain
\[
\begin{split}
A''(0) \leq \int_{\de E}\Big( & -\frac{H^{2}}{n-1}u^{-\gamma}-(n-1)\lambda u^{-\gamma}-\frac{1}{N}f_{\nu}^{2}u^{-\gamma} \\[1mm]
& +\gamma u^{-\gamma-1}u_{\nu}f_{\nu}-\gamma u^{-\gamma-2}u_\nu^2-\gamma u^{-\gamma-1}Hu_\nu \\[1mm]
& +(H-f_{\nu}+\gamma u^{-1}u_\nu)(-f_\nu u^{-\gamma}+Hu^{-\gamma}) \Big)e^{-f} \\
= \int_{\de E}\Big( & -\frac{H^{2}}{n-1}-(n-1)\lambda-\frac{1}{N}f_{\nu}^{2}+\gamma u^{-1}u_{\nu}f_{\nu}-\gamma u^{-2}u_\nu^2 \\
& -\gamma u^{-1}Hu_\nu+\big(H-f_{\nu}+\gamma u^{-1}u_\nu \big)\big(-f_\nu+H\big)\Big) u^{-\gamma}e^{-f}.
\end{split}
\]

\bigskip
\noindent
{\bf Step 2.} Verify \eqref{goal G}.
\bigskip

For convenience, we introduce $X=G'(v_{0})u^{\alpha-\gamma}e^{-kf}$ and $Y=u^{-1}u_\nu$. Then \eqref{G'} shows
\[
H = f_{\nu}+X-\gamma Y.
\]
We compute
\[
\begin{split}
G'(v_{0})V''(0) = {} & \int_{\de E}u^{\gamma-\alpha}e^{kf}X\cdot\Big(X-kf_{\nu}+(\alpha-2\gamma)Y\Big)u^{\alpha-2\gamma}e^{-(k+1)f} \\
= {} & \int_{\de E}\Big(X^{2}-kXf_{\nu}+(\alpha-2\gamma)XY\Big)u^{-\gamma}e^{-f}
\end{split}
\]
and
\[
\begin{split}
A''(0) \leq \int_{\de E}\bigg[ & -\frac{(f_{\nu}+X-\gamma Y)^{2}}{n-1}-(n-1)\lambda-\frac{1}{N}f_{\nu}^{2}+\gamma Yf_{\nu} \\
& -\gamma Y^{2}-\gamma Y(f_{\nu}+X-\gamma Y)+X(X-\gamma Y) \bigg]u^{-\gamma}e^{-f} \\
= \int_{\de E}\bigg[ & -(n-1)\lambda-\Big(\frac{1}{n-1}+\frac{1}{N}\Big)f_{\nu}^{2}-\frac{2}{n-1}Xf_{\nu}\\
&+\frac{2\gamma}{n-1}Yf_{\nu} +\Big(1-\frac{1}{n-1}\Big)X^{2}+\Big(\frac{2\gamma}{n-1}-2\gamma\Big)XY \\
& +\Big(\frac{n-2}{n-1}\gamma^{2}-\gamma\Big)Y^{2}\bigg]u^{-\gamma}e^{-f}.
\end{split}
\]
It then follows that
\[
\begin{split}
& A''(0)-G'(v_{0})V''(0) \\
\leq {} & \int_{\de E}\bigg[ -(n-1)\lambda-\frac{N+n-1}{N(n-1)}f_{\nu}^{2}+\Big(k-\frac{2}{n-1}\Big)Xf_{\nu}+\frac{2\gamma}{n-1}Yf_{\nu} \\
& \quad \quad \ \, -\frac{1}{n-1}X^{2}+\Big(\frac{2\gamma}{n-1}-\alpha\Big)XY+\Big(\frac{n-2}{n-1}\gamma^{2}-\gamma\Big)Y^{2}\bigg]u^{-\gamma}e^{-f}.
\end{split}
\]
Recalling $k=\frac{2}{n-1}$ and $\alpha=\frac{2\gamma}{n-1}$,
\[
\begin{split}
& A''(0)-G'(v_{0})V''(0) \\
\leq {} & \int_{\de E}\bigg[-(n-1)\lambda-\frac{1}{n-1}X^{2}-\frac{N+n-1}{N(n-1)}f_{\nu}^{2} \\
& \quad \quad \ \, +\frac{2\gamma}{n-1}Yf_{\nu}+\Big(\frac{n-2}{n-1}\gamma^{2}-\gamma\Big)Y^{2}\bigg]u^{-\gamma}e^{-f}.
\end{split}
\]
Using
\[
-\frac{N+n-1}{N(n-1)}f_{\nu}^{2}+\frac{2\gamma}{n-1}Yf_{\nu}
\leq  \frac{N\gamma^{2}}{(N+n-1)(n-1)} Y^2
\]
and $0\leq\gamma\leq\frac{N+n-1}{N+n-2}$, we obtain
\[
\begin{split}
& A''(0)-G'(v_{0})V''(0) \\
\leq {} & \int_{\de E}\bigg[-(n-1)\lambda-\frac{1}{n-1}X^{2}+\Big(\frac{N+n-2}{N+n-1}\gamma^{2}-\gamma\Big)Y^{2}\bigg]u^{-\gamma}e^{-f} \\
\leq {} & \int_{\de E}\Big(-(n-1)\lambda-\frac{1}{n-1}X^{2}\Big)u^{-\gamma}e^{-f}.
\end{split}
\]
Combining this with $V'(0)=\int_{\de E}u^{\alpha-\gamma} e^{-(k+1)f}$ and $X=G'(v_{0})u^{\alpha-\gamma}e^{-kf}$,
\[
\begin{split}
G''(v_{0}) = {} & \frac{A''(0)-G'(v_{0})V''(0)}{V'^{2}(0)} \\
\leq {} & -(n-1)\lambda\cdot\frac{\int_{\de E}u^{-\gamma}e^{-f}}{\left(\int_{\de E}u^{\alpha-\gamma}e^{-(k+1)f}\right)^{2}}
-\frac{G'^{2}(v_{0})}{n-1}\cdot\frac{\int_{\de E}u^{2\alpha-3\gamma}e^{-(2k+1)f}}{\left(\int_{\de E}u^{\alpha-\gamma}e^{-(k+1)f}\right)^{2}}.
\end{split}
\]
Recalling $\inf_{M}u=1$ and $\gamma\geq\alpha$,
\[
\begin{split}
\int_{\de E}u^{-\gamma}e^{-f}
& = \int_{\de E}u^{2\gamma-2\alpha}e^{2kf} \cdot u^{2\alpha-3\gamma}e^{-(2k+1)f} \\
& \geq e^{-2kF}\int_{\de E}u^{2\alpha-3\gamma}e^{-(2k+1)f}.
\end{split}
\]
It then follows that
\[
G''(v_{0}) \leq \left(-(n-1)\lambda e^{-2kF}
-\frac{G'^{2}(v_{0})}{n-1}\right)\cdot\frac{\int_{\de E}u^{2\alpha-3\gamma}e^{-(2k+1)f}}{\left(\int_{\de E}u^{\alpha-\gamma}e^{-(k+1)f}\right)^{2}}.
\]
Using the Cauchy-Schwarz inequality
\[
\begin{split}
\left(\int_{\de E}u^{\alpha-\gamma}e^{-(k+1)f}\right)^{2}
= {} & \left(\int_{\de E}u^{\alpha-\frac{3}{2}\gamma}e^{-k-\frac{1}{2}f} \cdot u^{\frac{1}{2}\gamma}e^{-\frac{1}{2}f}\right)^{2} \\
\leq {} & \left(\int_{\de E}u^{2\alpha-3\gamma}e^{-(2k+1)f}\right)\left(\int_{\de E}u^{\gamma}e^{-f}\right),
\end{split}
\]
we see that
\[
G''(v_{0}) \leq \left(-(n-1)\lambda e^{-2kF}
-\frac{G'^{2}(v_{0})}{n-1}\right)\cdot\frac{1}{\int_{\de E}u^{\gamma}e^{-f}}.
\]
Recalling $G(v_{0})=I(v_{0})=\int_{\de E}u^{\gamma}e^{-f}$, we obtain \eqref{goal G}.
\end{proof}

\begin{proof}[Proof of (2) in Theorem \ref{main theorem} when $3\leq n\leq7$]
Combining Lemma \ref{ODE comparison}, \eqref{ODE comparison condition a}, \eqref{ODE comparison condition b} and Lemma \ref{ODE comparison condition c}, the function $J:=e^{-\frac{F}{n}}I$ is a continuous function on $[0,V_{1})$ satisfies the assumptions (a)-(c) in Lemma \ref{ODE comparison} with $V_{1}=e^{-\frac{F}{n}}V_{0}$ and $\Lambda=\lambda e^{-(2k+\frac{2}{n})F}$. Then
\[
e^{-\frac{F}{n}}V_{0} \leq \big(\lambda e^{-\big(2k+\frac{2}{n}\big)F}\big)^{-\frac{n}{2}}\vol(\mathbb{S}^{n})
\]
and so
\[
V_{0} \leq e^{\big(kn+1+\frac{1}{n}\big)F}\lambda^{-\frac{n}{2}}\vol(\mathbb{S}^{n}).
\]
By $\inf_{M}u=1$ and $\alpha\geq0$, we have
\[
V_{0} = \int_{M}u^{\alpha}e^{-(k+1)f} \geq e^{-kF}\int_{M}e^{-f} = e^{-kF}\vol_{f}(M,g).
\]
Combining the above and recalling $k=\frac{2}{n-1}$,
\[
\vol_{f}(M,g) \leq e^{\frac{(n+1)(3n-1)}{n(n-1)}F} \cdot \lambda^{-\frac{n}{2}}\vol(\mathbb{S}^{n}).
\]
\end{proof}

\subsection{The case $n\geq8$}
In this case, $\de E$ may have singularities. The global volume comparison can be proved by modifying the argument (of the case $3\leq n\leq 7$) verbatim as in \cite[Proof of Lemma 2 ($n\geq8$) in Appendix A]{AX17}.


\begin{thebibliography}{99}

\bibitem{AX17} Antonelli, G.; Xu, K.; {\em New spectral Bishop-Gromov and Bonnet-Myers theorems and applications to isoperimetry}, preprint, arXiv:2405.08918.

\bibitem{Bray97} Bray, H., {\em The Penrose inequality in general relativity and volume comparison theorems involving scalar curvature}, Thesis (Ph.D.)--Stanford University
ProQuest LLC, Ann Arbor, MI, 1997. 103 pp.

\bibitem{BGLZ19} Bray, H.; Gui, F.; Liu, Z.; Zhang, Y., {\em Proof of Bishop's volume comparison theorem using singular soap bubbles}, preprint, arXiv:1903.12317.

\bibitem{BE85} Bakry, D.; \'{E}mery, M., {\em Diffusions hypercontractives}, S\'{e}minaire de probabilit\'{e}s, XIX, 1983/84, 177--206. Lecture Notes in Math., 1123, Springer-Verlag, Berlin, 1985.

\bibitem{BQ00} Bakry, D.; Qian, Z., {\em Some new results on eigenvectors via dimension, diameter, and Ricci curvature}, Adv. Math. {\bf 155} (2000), no. 1, 98--153.

\bibitem{BQ05} Bakry, D.; Qian, Z., {\em Volume comparison theorems without Jacobi fields}, Current trends in potential theory, 115--122. Theta Ser. Adv. Math., {\bf 4} Theta, Bucharest, 2005.

\bibitem{CLMS24} Chodosh, O.; Li, C.; Minter, P.; Stryker, D., {\em Stable minimal hypersurfaces in $\mathbf{R}^{5}$}, preprint, arXiv:2401.01492.

\bibitem{Gromov19} Gromov, M., {\em Four Lectures on Scalar Curvature}, preprint, arXiv:1908.10612.

\bibitem{Li05} Li, X.-D., {\em Liouville theorems for symmetric diffusion operators on complete Riemannian manifolds}, J. Math. Pures Appl. (9) {\bf 84} (2005), no. 10, 1295--1361.

\bibitem{Lichnerowicz70} Lichnerowicz, A., {\em Vari\'{e}t\'{e}s riemanniennes \`{a} tenseur C non n\'{e}gatif}, C. R. Acad. Sci. Paris S\'{e}r. A-B {\bf 271} (1970), A650--A653.

\bibitem{Lichnerowicz7172} Lichnerowicz, A., {\em Vari\'{e}t\'{e}s k\"{a}hl\'{e}riennes \`{a} premi\`{e}re classe de Chern non negative et vari\'{e}t\'{e}s riemanniennes \`{a} courbure de Ricci g\'{e}n\'{e}ralis\'{e}e non negative}, J. Differential Geometry {\bf 6} (1971/72), 47--94.

\bibitem{Limoncu12} Limoncu, M., {\em The Bakry-Emery Ricci tensor and its applications to some compactness theorems}, Math. Z. {\bf 271} (2012), no. 3--4, 715--722.

\bibitem{Lott03} Lott, J., {\em Some geometric properties of the Bakry-\'{E}mery-Ricci tensor}, Comment. Math. Helv. {\bf 78} (2003), no. 4, 865--883.

\bibitem{Maggi12} Maggi, F. {\em Sets of finite perimeter and geometric variational problems}, An introduction to geometric measure theory, Cambridge Stud. Adv. Math., 135, Cambridge University Press, Cambridge, 2012. xx+454 pp.

\bibitem{Morgan03} Morgan, F., {\em Regularity of isoperimetric hypersurfaces in Riemannian manifolds}, Trans. Amer. Math. Soc. {\bf 355} (2003), no. 12, 5041--5052.

\bibitem{Morgan09} Morgan, F., {\em Geometric measure theory. A beginner's guide}, Fourth edition, Elsevier/Academic Press, Amsterdam, 2009. viii+249 pp.

\bibitem{Perelman02} Perelman, G., {\em The entropy formula for the Ricci flow and its geometric applications}, preprint, arXiv:math/0211159.

\bibitem{Qian97} Qian, Z., {\em Estimates for weighted volumes and applications}, Quart. J. Math. Oxford Ser. (2) {\bf 48} (1997), no. 190, 235--242.

\bibitem{Tadano16} Tadano, H., {\em Remark on a diameter bound for complete Riemannian manifolds with positive Bakry-\'{E}mery Ricci curvature}, Differential Geom. Appl. {\bf 44} (2016), 136--143.

\bibitem{WW09} Wei, G.; Wylie, W., {\em Comparison geometry for the Bakry-Emery Ricci tensor}, J. Differential Geom. {\bf 83} (2009), no. 2, 377--405.

\bibitem{Xu23} Xu, K., {\em Dimension constraints in some problems involving intermediate curvature}, preprint, arXiv:2301.02730, to apprear in Trans. Amer. Math. Soc.

\bibitem{Zhu21} Zhu, J., {\em Width estimate and doubly warped product}, Trans. Amer. Math. Soc. {\bf 374} (2021), no. 2, 1497--1511.

\end{thebibliography}
\end{document}